\pdfoutput=1
\documentclass{amsart}
\usepackage[all]{xy}
\usepackage{tikz,graphicx}
\usepackage{amssymb}

\newtheorem*{introthm}{Theorem}

\newtheorem{theorem}{Theorem}[section]
\newtheorem{lemma}[theorem]{Lemma}

\newtheorem{proposition}[theorem]{Proposition}
\newtheorem{corollary}[theorem]{Corollary}
\theoremstyle{definition}
\newtheorem{definition}[theorem]{Definition}

\newtheorem{remark}[theorem]{Remark}

\def\Eff{{\rm Eff}}

\def\Nef{{\rm Nef}}

\def\SAmple{{\rm SAmple}}

\def\cc{{\mathbb C}}

\def\rr{{\mathbb R}}
\def\nn{{\mathbb N}}
\def\qq{{\mathbb Q}}
\def\pp{{\mathbb P}}

\def\Exc{\operatorname{Exc}}
\def\Aut{\operatorname{Aut}}
\def\Sym{\operatorname{Sym}}

\def\Pic{\operatorname{Pic}}
\def\chull{\operatorname{Chull}}

\def\Osh{{\mathcal O}}

\begin{document}

\title{Cox rings of surfaces and the anticanonical Iitaka dimension}

\author{Michela Artebani}
\address{
Departamento de Matem\'atica, \newline
Universidad de Concepci\'on, \newline
Casilla 160-C,
Concepci\'on, Chile}
\email{martebani@udec.cl}

\author{Antonio Laface}
\address{
Departamento de Matem\'atica, \newline
Universidad de Concepci\'on, \newline
Casilla 160-C,
Concepci\'on, Chile}
\email{alaface@udec.cl}

\thanks{The first author has been partially 
supported by Proyecto FONDECYT Regular 2009, N. 1090069.
The second author has been partially supported 
by Proyecto FONDECYT Regular 2008, N. 1080403.}

\maketitle

\begin{abstract}
In this paper we investigate the relation between the finite generation of the Cox ring $R(X)$ of a smooth projective surface $X$ and its anticanonical Iitaka dimension $\kappa(-K_X)$.
\end{abstract}

\section*{Introduction}
 This paper discusses the problem of deciding which smooth projective surfaces over the complex numbers have finitely generated Cox ring.
 More precisely, if the Picard group $\Pic(X)$ of such a surface $X$ is finitely generated, or equivalently $q(X)=0$, then the {\em Cox ring} of $X$ is:
\[
R(X):=\bigoplus_{D\in \Pic(X)} H^0(X,\mathcal O_X(D)).
\]
The Cox ring of a surface $X$ is known to be a polynomial ring if and only if $X$ is toric \cite{cox, hk}. In \cite[Cor.~5.1]{to2} Totaro proved that the Cox ring of a klt Calabi-Yau pair of dimension two over $\cc$ is finitely generated if and only if its effective cone is rational polyhedral. 
In the recent paper ~\cite{tvv} by Testa, V\'arilly-Alvarado and Velasco, the authors prove that if $X$ is a smooth rational surface with $-K_X$ big, i.e.~such that $\dim\varphi_{|-nK_X|}(X)=2$ for $n$ big enough, then 
$R(X)$ admits a finite number of generators. These results motivate the research for a relation between the {\em anticanonical Iitaka dimension} of $X$ (see \cite[Def.~2.1.3]{laz}):
\[
\kappa(-K_X) := \max\{\dim \varphi_{|-nK_X|}(X) : n\in\nn\},
\]
whose values can be $2$, $1$, $0$ and $-\infty$, and the finite generation of $R(X)$.
Let $\Eff(X)$  be the convex cone generated by classes of effective divisors in $N^1(X)=\Pic(X)\otimes \rr/\equiv$, where $\equiv$ is numerical equivalence.
Our first result is the following.
\begin{introthm}
Let $X$ be a smooth rational surface with $\kappa(-K_X)=1$. Then the following are equivalent:
\begin{enumerate}
\item the effective cone $\Eff(X)$ is rational polyhedral;
\item the Cox ring $R(X)$ is finitely generated;
\item $X$ contains finitely many $(-1)$-curves.
\end{enumerate}
\end{introthm}
Moreover we prove that, if $\kappa(-K_X)=1$, then the effective cone of $X$ is rational polyhedral if and only if the same is true for its relative minimal model.
This allows us to show (Corollary \ref{infty}) that there exist surfaces with finitely generated Cox ring and anticanonical Iitaka dimension $1$ of any Picard number $\geq 9$.

In case $-K_X$ is nef we prove the following result, which relies on Nikulin's description of surfaces with rational polyhedral effective cone \cite[Ex.~1.4.1]{n1}.

\begin{introthm}
Let $X$ be a smooth projective surface with $q(X)=0$ and $-K_X$ 
nef. Then $R(X)$ is finitely generated
if and only if one of the following holds:
\begin{enumerate}
\item  $X$ is the minimal resolution 
of singularities of a Del Pezzo surface with Du Val 
singularities;
\item  $\varphi_{|-mK_X|}$ is an elliptic fibration for some $m>0$ and the Mordell-Weil group of the Jacobian fibration of $\varphi_{|-mK_X|}$ is finite;
\item $X$ is either a K3-surface or an Enriques surface with finite automorphism group $\Aut(X)$.
\end{enumerate}
\end{introthm}
Surfaces of type (i) are classically known by \cite{na}, surfaces in (ii)  can be classified by means of~\cite{d, cd} (for $m=1$) and Ogg-Shafarevich theory \cite{o,sh}, while surfaces of type (iii)  have been classified in a  series of papers by Nikulin and Kond\=o (see ~\cite[Ex.~1.4.1]{n1} for precise references).

The paper is structured as follows. In Section 1 we introduce four cones in $N^1(X)$: the effective cone, the closed light cone, the nef cone and the semiample cone. Section 2 deals with the structure of the effective cone of rational surfaces with $\kappa(-K_X)\geq 0$. Our main result here is the following.
\begin{introthm}
Let $X$ be a smooth rational surface with $\kappa(-K_X)\geq 0$ and $\rho(X)\geq 3$, then $\overline{\Eff(X)}=\overline{E(X)}$, where $E(X)$ is the cone generated by classes of integral curves of $X$ with negative self-intersection.
\end{introthm}
In Section 3 we prove that on a smooth rational surface $X$ with $\kappa(-K_X)=1$  every nef $\qq$-divisor is semiample.
The problem of finite generation of the Cox ring of such surfaces is considered in Section 4.
Section 5 is devoted to the problem of finite generation of $R(X)$ under the hypothesis $-K_X$ nef. 
Finally, Section 6 shows an example of a non-rational surface $X$ with $\rho(X)=2$ and rational polyhedral effective cone whose Cox ring does not admit a finite set of generators.\\

\noindent{\em Acknowledgments:} It is a pleasure to thank Tommaso de Fernex and Damiano Testa for several useful comments which helped us to improve and clarify this work.
We are also grateful to Jinhyung Park, who informed us about a mistake in the previous version of Lemma 4.4 and suggested how to fix it.
\section{Basic setup}

In what follows $X$ will denote a smooth projective surface defined over the complex numbers. 
Given a divisor $D$ of $X$ we will adopt the short notation $H^i(D)$ for the cohomology group $H^i(X,\Osh_X(D))$ and we will denote its dimension by $h^i(D)$.
Also, we will denote by $\equiv$ the numerical equivalence between divisors, by $[D]$ the class of $D$ in $N^1(X)=\Pic(X)\otimes \rr/\equiv$ and by $|D|$ the complete linear series associated to $D$. Observe that $N^1(X)=\Pic(X)\otimes \rr$ if $q(X)=0$, in particular this is true if $X$ is a rational surface.
We recall that the {\em effective cone} of an algebraic surface $X$ is defined as:
\[
\Eff(X) := \{\sum_i a_i[D_i] : D_i\text{ is an effective divisor}, a_i\in \rr_{\geq 0}\}.
\]
The {\em closed light cone} of $X$ is the cone of classes with non-negative self-intersection:
\[
L(X):=\{[D]\in N^1(X) : D^2 \geq 0\}.
\]
We define $L_a(X)$ to be the half-cone of $L(X)$ which contains an ample class.
In what follows we will say that a cone of $N^1(X)$ is {\em polyhedral} if it is generated by finitely many vectors. In particular a polyhedral cone is closed.
We start proving the following (see also~\cite[\S1]{n1}).
\begin{proposition}\label{eff-rays}
Let $X$ be a smooth projective surface such that $\rho(X)\geq 3$ and $\Eff(X)$ is  polyhedral. Then
\[
\Eff(X) = \sum_{[E]\in\Exc(X)} \rr_+ \cdot [E]
\]
where $\Exc(X)$ is the set of classes of integral curves $E$ of $X$ with $E^2<0$.  
\end{proposition}
\begin{proof}
This is a consequence of the following observation:  by Riemann-Roch theorem the interior of  $L_a(X)$ is contained in $ \Eff(X)$. Since the effective cone is polyhedral, then it is closed, so that $L_a(X)\subset \Eff(X)$.
Since $\rho=\rho(X)\geq 3$, then $\partial L_a(X)$ is circular because the intersection form is hyperbolic with signature $(1,\rho-1)$ by the Hodge index theorem.
Thus an element of $\partial {L_a}(X)$ can not be an extremal ray of $\Eff(X)$, since otherwise $\Eff(X)$ would not be polyhedral in a neighbourhood of that ray, giving a contradiction.
\end{proof}

Another important cone associated to $X$ is the cone of numerically effective divisors, or simply the {\em nef cone}:
\[
\Nef(X) := \{[D]\in N^1(X) : D\cdot E\geq 0\text{ for any }[E]\in\Eff(X)\}.
\] 
This cone is the dual of the effective cone with respect to the intersection form on the surface $X$.  
Finally, if $q(X)=0$, we define the {\em semiample cone} $\SAmple(X)$ to be the cone spanned by the classes of semiample divisors, where $D$ is semiample if $|nD|$ is base point free for some $n>0$ (see \cite[Def. 1.1.10, 2.1.26]{laz}).
\begin{proposition}\label{eff}
Let $X$ be a smooth projective surface with $q(X)=0$. We have the following inclusions:
\[
\SAmple(X)\subset\Nef(X)\subset\overline{\Eff(X)}.
\]\end{proposition}
\begin{proof}
The second inclusion is due to the fact that the nef cone is the closure of the ample cone~\cite[Thm.~1.4.23]{laz}. For the first inclusion observe that, if $D$ is semiample, then $\varphi_{|nD|}: X\to\pp^r$ is a morphism for $n$ big enough. Thus, if $E$ is an effective divisor, then $nD\cdot E=\deg (\varphi_{|nD|}^*\Osh_{\pp^r}(1)_{|E})\geq 0$, so that $D$ is nef.
\end{proof}

The following will be useful in the next sections.
\begin{proposition}\label{fib}
Let $X$ be a smooth projective surface with $q(X)=0$ and let $M$ be a non-trivial effective 
divisor of $X$ such that $|M|$ does not contain fixed components and
$M^2=0$. Then $M\sim aD$ with $D$ smooth and integral, $h^0(D)=2$ and
$H^0(M)\cong\Sym^a H^0(D)$.
\end{proposition}
\begin{proof}
The linear series $|M|$ is base point free, since otherwise two of its
distinct elements would intersect at the base points giving $M^2>0$,
which is a contradiction.

Let $\varphi_{|M|}: X\to B\subset \pp^n$ be the morphism defined by
$|M|$. Since holomorphic 1-forms of $B$ pull-back to $X$ we have
$q(X)\geq p_a(B)$. Thus $q(X)=0$ implies that  $B$ is smooth and
rational. Consider the Stein factorization of $\varphi_{|M|}$:
\[
\xymatrix@1{
X\ar[rr]^-{\varphi_{|M|}}\ar[rd]_-{f} & & \pp^n\\
& \pp^1\ar[ur]_{\nu} &
}
\]
where $f$ is a morphism with connected fibers and $\nu$ is a finite map.
If $a:=\deg(\nu)\deg(\nu(\pp^1))$, then $M\sim aD$, where
$\Osh_X(D)=f^*\Osh_{\pp^1}(1)$. Since $h^0(D)\geq 2$, then
$n+1=h^0(M)=h^0(aD)\geq a+1$. On the other hand
$a\geq\deg(\nu(\pp^1))\geq n$ since the curve $\nu(\pp^1)$ is
non-degenerate.
Thus $n=a$, $h^0(D)=2$ and the map $\nu$ is the $a$-Veronese embedding
of $\pp^1$.
Since $f$ has connected fibers, then $D$ is connected so that, by
Bertini's second theorem~\cite{is}, the general element of $|D|$ is
smooth.
\end{proof}

\section{The structure of the effective cone}

Let $X$ be a projective surface with $q(X)=0$ and $\kappa(-K_X)\geq 0$. Observe that in this case, either $K_X$ is numerically trivial, or $X$ is rational by Castelnuovo's rationality criterion~\cite[Thm.~3.4, VI]{bpv}.
We consider the problem of determining under which hypothesis the effective cone of $X$ is rational polyhedral. 

Let $L_a(X)$ be the component of the closed light cone which contains the ample cone, as in the previous section,  and let $E(X)$ be the convex cone generated by the classes of curves in $X$ with negative self-intersection.
Given a cone $\sigma\subseteq N^1(X)$ we will adopt the following notation:
\[
\sigma_{\geq 0} = \{[D]\in\sigma :D\cdot K_X\geq 0\},
\qquad
\sigma_{\leq 0} = \{[D]\in\sigma :D\cdot K_X\leq 0\}.
\]
The cones $\sigma_{> 0}$ and $\sigma_{< 0}$ are defined in a similar way.  
\begin{theorem}\label{effc}
If $X$ is a smooth rational surface with $\kappa(-K_X)\geq 0$ and $\rho(X)\geq 3$, then
\[
\overline{\Eff(X)}=\overline{E(X)}.
\]
\end{theorem}
\begin{proof}

We divide the proof in three steps.
\vspace{0.1cm}

\noindent {\em Step 1.} We prove that 
\[
\chull(\overline{E(X)},L_a(X))=\overline{\Eff(X)}.
\]
If the sets are distinct then, since both are closed, there exists a  class $[D]\in \Eff(X)\backslash \chull(\overline {E(X)},L_a(X))$. In particular $D^2<0$, so that $|D|$ contains at least a negative curve as a fixed component.
Thus $D\sim D_1+D_2$, where $D_1, D_2$ are effective and $D_1$ consists of all the negative curves contained in the fixed part of $|D|$. Since $[D_1]\in E(X)$, then $[D_2]\not \in \chull(\overline{E(X)},L_a(X))$, so that  $D_2^2<0$. Then there is still a negative curve in the fixed part of $|D_2|$ and thus in the fixed part of $|D|$, giving a contradiction.
\vspace{0.1cm}

\noindent {\em Step 2.} 
We now prove that
\[
L_a(X)_{\geq 0}\subset\chull(\overline{E(X)},L_a(X)_{\leq 0}).
\]
If the interior of $L_a(X)_{\geq 0}$ is empty, then $L_a(X) = L_a(X)_{\leq 0}$, so we get the claim by Step 1.
Otherwise, let $[D]$ be a class in the interior of $L_a(X)_{\geq 0}$, i.e. $D^2>0$ and $D\cdot K_X>0$. 
Observe that for some positive integers $n,m$, the multiples $nD$, $-mK_X$ are effective since $D^2>0$ and $\kappa(-K_X)\geq 0$.
Thus, since $nD\cdot (-mK_X)<0$, then $|nD|$ contains at least a negative curve in its fixed locus. Let $D_1$ be given by all curves with negative self-intersection in the fixed locus of $|nD|$.  Then the divisor $D_2=D-D_1$ is nef, so that $D_2^2\geq 0$ and $D_2\cdot (-K_X)\geq 0$.
Thus $D_2\in L_a(X)_{\leq 0}$.
Together with Step 1, this gives:
\[
\chull(\overline{E(X)},L_a(X)_{\leq 0})=\overline{\Eff(X)}.
\]

\noindent {\em Step 3.} 
Let $l_+\in\partial L_a(X)_{>0}$,  so that $l_+^2=0$ and $l_+\cdot K_X>0$. By Step 2 we have $l_+=e+l_-$, where $e\in\overline{E(X)}$ and $l_-\in L_a(X)_{\leq 0}$. Assume that $l_+$ is an extremal ray of $\overline{\Eff(X)}$, then  $l_+=e$ is an extremal ray of $\overline{E(X)}$. Observe that $\overline{E(X)}_{>0}$ and $E(X)_{>0}$ have the same extremal rays, since the last convex set contains a finite number of extremal rays, which are classes of curves contained in the base locus of an effective multiple of $-K_X$. Then $l_+$ is an extremal ray of  $E(X)_{>0}$, so that $l_+^2<0$, which is a contradiction.

Assume now that $z_-\in\partial L_a(X)_{<0}$, so that $z_-^2=0$ and $z_-\cdot K_X=-\epsilon<0$.  Since $X$ is rational and $\rho(X)\geq 3$, then by the Cone Theorem~\cite[Thm.~1.5.33]{laz} and \cite[Lemma~6.2]{deb} 
we have:
\[
\overline{\Eff(X)}=\overline{\Eff(X)}_{\geq 0}+\sum_i \rr_+\cdot [D_i],
\]
for countably many $[D_i]\in E(X)$ which can have accumulation points only on the hyperplane $K_X^{\perp}$. This implies that $z_-=z_++e'$, where $z_+\in \overline{\Eff(X)}_{\geq 0}$ and  $e'\in \overline{E(X)}_{\leq 0}$.
If $z_-$ is an extremal ray of $\overline{\Eff(X)}$, then  $z_-=e'$ is an extremal ray of $\overline{E(X)}_{<-\epsilon/2}=E(X)_{<-\epsilon/2}$ by the Cone theorem. Then, since $z_-$ is an extremal ray, we get $z_-^2<0$, which is a contradiction.

We proved that neither $\partial L_a(X)_{<0}$ nor $\partial L_a(X)_{>0}$ can be at the boundary of the effective cone, thus $L_a(X)\subset\overline{E(X)}$.
\end{proof}

\begin{corollary}\label{effneg}
Let $X$ be a smooth rational surface with $\kappa(-K_X)\geq 0$, then the following are equivalent
\begin{enumerate}
\item $\Eff(X)$ is rational polyhedral,
\item $X$ contains finitely many $(-1)$ and $(-2)$-curves.
\end{enumerate}
\end{corollary}
\begin{proof}
If $\rho(X)\leq 2$,  then $X$ is a toric surface, either the projective plane or a Hirzebruch surface, so that both the conditions are obviously satisfied. So now we assume that $\rho(X)\geq 3$.
The implication $(i)\Rightarrow (ii)$  is given by Proposition~\ref{eff-rays}, since the classes of $(-1)$ and $(-2)$-curves span extremal rays of the effective cone.
To prove the converse, by Theorem \ref{effc} it is enough to prove that $E(X)$ is rational polyhedral, or equivalently that $X$ contains finitely many classes of integral curves with negative self-intersection. The set of such classes in ${E(X)}_{>0}$ is finite, since the corresponding curves belong to the base locus of an effective multiple of $-K_X$.
By the Cone Theorem and \cite[Lemma~6.2]{deb} the curves with negative self-intersection with classes in ${E(X)}_{<0}$ are rational, thus they are $(-1)$-curves by adjunction formula. 
Finally, a curve with negative self-intersection and orthogonal to $K_X$ is a $(-2)$-curve. 
Thus we conclude by observing that $E(X)\subseteq\Eff(X)\subseteq\overline{\Eff(X)}=\overline{E(X)}=E(X)$.
\end{proof}

\begin{remark}\label{k2}
If $\kappa(-K_X)=2$, then $\Eff(X)$ is rational polyhedral~\cite[Prop.~3.3]{na}.
The effective cone of a smooth rational surface $X$ with  $\kappa(-K_X)=1$ is not necessarily polyhedral. Consider as an example the blow-up $X$ of $\pp^2$ at the nine intersection points $\{p_1,\dots,p_9\}=C_1\cap C_2$ of two general plane cubics. Due to the generality assumption on the $C_i$'s, there are no reducible elements in the linear series $|-K_X|$, so that all the fibers of $\varphi_{|-K_X|}: X\to\pp^1$ are integral. 
Hence the class $-K_X$ is an extremal ray of the effective cone. By Proposition~\ref{eff-rays} and the fact that $K_X^2=0$, we deduce that $\Eff(X)$ is not polyhedral (see also~\cite[Cor. 3.2]{to1}).
\end{remark}

\section{The nef and the semiample cones} 
In what follows we will make use of the {\em Zariski decomposition} of a pseudoeffective divisor (see~\cite[Thm. 2.3.19]{laz}):  a pseudoeffective divisor $D$ can be written uniquely as a sum $D=N+P$, where $N$ and $P$ are $\qq$-divisors such that $P$ is nef, $N$ is effective and the intersection matrix of its components is negative definite, $P$ is orthogonal to each component of $N$. The divisors $P$ and $N$ are called the positive and negative part of $D$ respectively.

In what follows, we will denote by $\kappa(D)$ the Iitaka dimension of a divisor $D$ (see  \cite[Def.~2.1.3]{laz}).
\begin{lemma}\label{genere}
Let $X$ be a smooth rational surface with $\kappa(-K_X)\geq 1$. 
If $E$ is an effective divisor of $X$ such that the intersection matrix on its integral components is negative definite, then $p_a(E)\leq 0$.
\end{lemma}
\begin{proof}
We begin proving that $h^0(K_X+E)=0$. Observe that $\kappa(E)=0$ since the intersection form on its components is negative definite.
If $Z=K_X+E$ is an effective divisor, then $0=\kappa(E)=\kappa(-K_X+Z)\geq 1$, which is a contradiction.

Consider now the exact sequence of sheaves:
\[
\xymatrix@1{
0\ar[r] & \Osh_X(-E)\ar[r] & \Osh_X\ar[r] & \Osh_E\ar[r] & 0.
}
\]
Taking cohomology and using the fact that $h^1(\Osh_X)=h^2(\Osh_X)=0$ because $X$ is rational, we get $h^1(\Osh_E)=h^2(-E)$. The last is equal to $h^0(K_X+E)$ by the Serre's duality theorem. Now, from what proved before, we deduce that $h^1(\Osh_E)=0$ so that $p_a(E)\leq 0$, which proves the claim.
\end{proof}

\begin{lemma}\label{nef+big}
Let $X$ be a smooth rational surface with $\kappa(-K_X)\geq 1$ and let $L$ be a nef divisor with $\kappa(L)=2$. Then $L$ is semiample.
\end{lemma}
\begin{proof}
We follow the proof of~\cite[Lemma~2.6]{tvv}. Let $\Delta$ be the union of all integral curves orthogonal to $L$. Since $L^2>0$ then, by the Hodge index theorem, the restriction to $\Delta$ of the intersection form of $X$ is negative definite.
Moreover, by Lemma~\ref{genere}, we have that $p_a(E)\leq 0$ for any effective divisor supported on $\Delta$. 
Thus we can apply Artin's contractability criterion \cite[Thm. 2.3]{art} to $\Delta$. So, there exists a normal projective surface $Y$ and a birational morphism $\psi: X\to Y$ which contracts only the 
connected components of $\Delta$. Hence, by \cite[Cor. 2.6]{art}, $L$ is linearly equivalent to a divisor $L'$ whose support is disjoint from $\Delta$, and hence $L$ is the pullback of a Cartier divisor on $Y$. By the Nakai-Moishezon criterion $L'$ is ample, so that $L$ is semiample.
\end{proof}

\begin{lemma}\label{k=1}
Let $X$ be a smooth algebraic surface with $q(X)=0$ and let $L$ be a nef divisor with $\kappa(L)=1$. Then $L$ is semiample and $L^2=0$.
\end{lemma}
\begin{proof}
First of all observe that since $L$ is nef with $\kappa(L)=1$, then $L^2=0$.
Let $n\in \nn$ such that $h^0(nL)>1$. If $B$ is the fixed part of $|nL|$, then $\kappa(B)\leq 1$ so that $B^2\leq 0$. Since $nL-B$ is nef with $1\leq \kappa(nL-B)\leq \kappa(nL)=1$, then 
$(nL-B)^2=0$. This implies that $B^2=B\cdot L=0$. The restriction of the intersection 
form to the space spanned by $[B]$ and $[L]$ is null. By the Hodge index theorem this implies that $[B]=[\alpha L]$ for some $\alpha\in\qq_{>0}$. Thus the base locus of $|(n-\alpha)L|$ is $0$-dimensional and, since $L^2=0$, we see that it is actually empty.
This proves the claim.
\end{proof}
  
We are finally ready to prove the main theorem of this section.

\begin{theorem}\label{nef-big}
Let $X$ be a smooth rational surface with $\kappa(-K_X)\geq 1$. 
Then any nef $\qq$-divisor is semiample.  
  \end{theorem}
\begin{proof}
If $L$ is a nef divisor of $X$, then $L^2\geq 0$.
If $L^2>0$ or $K_X\cdot L<0$, then $h^0(L)\geq 2$ by Riemann-Roch formula so that $\kappa(L)\geq 1$ and we conclude by Lemma~\ref{nef+big} and~\ref{k=1}.
We assume now that $L^2=K_X\cdot L=0$ and let $-K_X\sim N+P$ be the Zariski decomposition of $-K_X$. 
Then $P\cdot L=0$ because $P$, $N$ are effective and $L$ is nef. 
The restriction of the intersection form of $\Pic(X)$ to the space spanned by $[P]$ and $[L]$ is null. Thus $L\sim mP$ for some $m\in\qq_{\geq 0}$. Hence $\kappa(P)=\kappa(-K_X)\geq 1$ and  $L$ is semiample by Lemma~\ref{nef+big} and~\ref{k=1}.
\end{proof}

\begin{remark}
Observe that if $\kappa(-K_X)=0$ and the positive part $P$ of the Zariski decomposition of $-K_X$ is non-trivial, then the nef and the semiample cone of $X$ do not coincide. This implies that the Cox ring $R(X)$ is not finitely generated by~\cite[Cor. 2.6]{ahl}. An easy example of such surfaces is given by the blow-up of $\pp^2$ at $9$ points in very general position. An example with $\Eff(X)$ rational polyhedral is given in~\cite[Ex.~1.4.1]{n1}.
\end{remark}

\section{Cox rings of rational surfaces with $\kappa(-K_X)=1$}

We consider the problem of the finite generation of Cox rings of smooth rational surfaces with $\kappa(-K_X)=1$.

\begin{proposition}\label{-K}
Let $X$ be a smooth rational surface with $\kappa(-K_X)=1$ and let
\[
-K_X\sim N+P
\]
be the Zariski decomposition of $-K_X$. Then $P\sim aC$ for some $a\in\qq_{>0}$, where $C$ is a smooth elliptic curve with $C^2=0$ and $h^0(C)=2$.
 \end{proposition}
\begin{proof}
Since $P$ is nef and $\kappa(P)=1$, then $P$ is semiample by Lemma~\ref{k=1}.
By Proposition~\ref{fib} we have $P\sim aC$ for some smooth integral curve $C$ with $C^2=0$ and $h^0(C)=2$.
By the genus formula and $-K_X\cdot P=0$ we get $2g(C)-2=C\cdot K_X=0$.
\end{proof}

\begin{theorem}\label{equiv}
Let $X$ be a smooth rational surface with $\kappa(-K_X)=1$. Then the following are equivalent:
\begin{enumerate}
\item the effective cone $\Eff(X)$ is rational polyhedral;
\item the Cox ring $R(X)$ is finitely generated;
\item $X$ contains finitely many $(-1)$-curves.
\end{enumerate}
\end{theorem} 
\begin{proof}
(i)$\Rightarrow$(ii): by~\cite[Cor. 2.6]{ahl} it is enough to prove that the nef and semiample cone of $X$ coincide and this is proved in Theorem~\ref{nef-big}.

(ii)$\Rightarrow$(iii): if $C$ is a $(-1)$-curve and $x\in H^0(C)$, write $x=\sum_im_i$, where $m_i$ are monomials in the generators of $R(X)$. If $D_i$ is the zero locus of $m_i$, then $C\sim D_i$. Since $C$ is integral with $C^2<0$, then $D_i=C$ so that $m_i=\alpha_i x$, with $\alpha_i\in\cc$. Hence $x$ appears in any set of homogeneous generators of $R(X)$. Since $R(X)$ is finitely generated, this implies that there are finitely many $(-1)$-curves.

(iii)$\Rightarrow$(i):
We can assume that $\rho(X)\geq 9$ since otherwise $\kappa(-K_X)=2$.
By Corollary~\ref{effneg} it is enough to prove that $X$ contains finitely many $(-2)$-curves.

Let $-K_X\sim N+aC$ be the Zariski decomposition of $-K_X$ as given in Proposition~\ref{-K}. 
If $E$ is a $(-2)$-curve, then $-K_X\cdot E=0$, so that either $E$ is contained in the support of $N$  or $E\cdot C=0$. In the last case $E$ is contracted by the morphism $\varphi_{|C|}: X\to\pp^1$, which is a fibration by Proposition~\ref{-K}.
Since the support of $N$ contains a finite number of prime divisors and $\varphi_{|C|}$ has finitely many reducible fibers, then $X$ contains finitely many $(-2)$-curves.
\end{proof}

Let $X$ be a smooth rational surface with $\kappa(-K_X)=1$ and let $-K_X\sim P+N$ be the Zariski decomposition. By Proposition~\ref{-K} a multiple of $P$ defines an elliptic fibration. We have the following commutative diagram:
\[
\xymatrix{
& X\ar[ld]_{\pi} \ar[rd]^{\varphi_{|rP|}}& \\
Y\ar[rr]_{\varphi_{|-mK_Y|}} && \pp^1
}
\]
where $\pi$ is the blow-down map of all the $(-1)$-curves contained in the fibers of $\varphi_{|rP|}$ and $m$ is the smallest positive integer such that $h^0(-mK_Y)=2$.
The surface $Y$ is the {\em relative minimal model} and its anticanonical divisor $-K_Y$ is nef. If $m=1$, then $Y$ is called a {\em jacobian elliptic surface}.

Let  $\varphi:X\to \pp^1$ be a fibration with connected fibers and let 
Let $L\sim aC$, with $a\in \qq$, where $C$ is an effective divisor whose support 
is contained in the fibers of  $\varphi$. 
\begin{definition}
The multiplicity of $L$ at $p\in X$ is:
\[
\mu(L,p):=a\mu(C_p,p),
\]
where $C_p\in |C|$ is the unique curve through $p$,
and $\mu(L,p)=0$ if there is no such curve.
\end{definition}
 
\begin{lemma}\label{mu}
Let $\pi: \tilde{X}\to X$ be the blow-up at a point $p$ of a smooth rational surface with $\kappa(-K_X)=1$. Then $\kappa(-K_{\tilde{X}})=1$ if and only if $\mu(-K_X,p)>1$.
In this case we have the Zariski decomposition $-K_{\tilde{X}}\sim 
\tilde P+\tilde N$ with $\tilde P=\pi^*P$ if $\mu(N,p)\geq 1$ and otherwise
\[
\tilde P= \frac{\mu(-K_X,p)-1}{\mu(P,p) }\,\pi^*P.
\]
\end{lemma}
\begin{proof}
Let $-K_X\sim P+N$ be the Zariski decomposition and
$\mu_P, \mu_N$ be the multiplicities of $P$ and $N$ in $p$.
If $\mu_N\geq 1$ then up to multiples $\pi^*N-E$ is linearly equivalent to an effective divisor which is orthogonal to $\pi^*P$ and such that the intersection matrix on the components of its support is negative definite.
Thus the Zariski decomposition of $-K_{\tilde X}$ is:
 $$-K_{\tilde X}\sim  \pi^*P+(\pi^*N-E).$$
 Otherwise, if $\mu_N<1$ and $\mu_P+\mu_N\geq 1$, consider the following decomposition
 $$-K_{\tilde X}=\frac{\mu_P+\mu_N-1}{\mu_P}\pi^*P+(\pi^*N-\mu_NE)+(1-\mu_N)\left(\frac{1}{\mu_P}\pi^*P-E\right).$$
Observe that the second and third term in the sum are linearly equivalent, up to multiples, to effective divisors whose supports are properly contained in the fibers of $\pi^*P$. In particular, for $\frac{1}{\mu_P}\pi^*P-E$, such support is contained in the unique fiber containing $E$. 
Since $E$ is not contained in any of the two supports, by the definition of $\mu_P$ and $\mu_N$,
then their sum can not contain any positive rational multiple of $\pi^*P$, so the intersection form on its components is negative definite by~\cite[Lemma 8.2, III]{bpv}. 
Since the two divisors are clearly orthogonal to the first term in the sum, then their sum is the negative part of the Zariski decomposition of $-K_{\tilde X}$. 
Finally, if $\mu_P+\mu_N>1$, then $\kappa(-K_{\tilde X})=1$ .

Conversely, observe that if $\mu(-K_X,p)= 1$, then $\kappa(-K_{\tilde X})=0$ by the previous decomposition.
The same clearly holds if  $\mu(-K_X,p)< 1$, since in this case any multiple of $-K_{\tilde X}=\pi^*(-K_X)-E$ is not linearly equivalent to an effective divisor.
\end{proof}

\begin{remark}\label{rem} Observe that if $\mu_P=\mu(P,p)>1$ we can always decompose $-K_{\tilde X}$ as $\big(1-\frac{1}{\mu_P}\big)\pi^*P+(\frac{1}{\mu_P}\pi^*P+\pi^*N-E)$,
where the two terms are effective, orthogonal divisors and the second one is negative semidefinite.

\end{remark}

\begin{proposition}\label{exist}
Let $n\geq 10$ be an integer. There exists a smooth rational surface $X$ with $\kappa(-K_X)=1$ and $\rho(X)=n$.
\end{proposition}
\begin{proof}
Let $X_0$ be a smooth rational surface such that $\varphi_{|-K_{X_0}|}: X\to\pp^1$ is an elliptic fibration which admits a fiber $P_0$ with a triple point. For example, $X_0$ can be the blow-up of $\pp^2$ at the nine base points of the pencil $xy(x+y)+t(x^3+y^3+z^3)=0$.
We construct a sequence of blow-ups $\pi_i: X_i\to X_{i-1}$, where $\pi_i$ is the blow-up of $X_{i-1}$ at a point $p_{i-1}$ chosen in the following way.
Let $p_0$ be the triple point of $P_0$ and let $p_i$ lie on the intersection of the exceptional divisor $E_{i-1}=\pi_i^{-1}(p_{i-1})$ and the strict transform $E_{i-2}'$ of $E_{i-2}$ (see Figure~\ref{blow-up} below).
Let $-K_{X_i}\sim P_i+N_i$ be the decomposition given in Remark \ref{rem} with $P_{i}\sim (1-\frac{1}{\mu_{i-1}})\pi^*_{i}(P_{i-1})$ where $\mu_{i-1}=\mu(P_{i-1},p_{i-1})$. This gives the formula:
\[
P_{i}\sim \prod_{k=0}^{i-1}(1-\frac{1}{\mu_k})~\phi^*_{i}P_0,
\]
where $\phi_i=\pi_1\circ\cdots\circ\pi_i: X_i\to X_0$ is the blow-down map. In this way we can recursively calculate $\mu_{i}$ obtaining:
\begin{equation}\label{mui}
\mu_{i}=\prod_{k=0}^{i-1}(1-\frac{1}{\mu_k})~\mu(\phi_i^*P_0,p_{i}).
\end{equation}
Observe that $\mu_0=3$ since $p_0$ is a triple point of a fiber of the elliptic fibration $\varphi_{|P_0|}$. Let $a_i:=\mu(\phi_i^*P_0,p_i)$, then $a_0=3$, $a_1=4$ and
\[
a_i=a_{i-1}+a_{i-2}.
\]
To see this observe that $\phi_i^*P_0$ contains $E_{i-1}$ with multiplicity $a_{i-1}$ and the strict transform of $E_{i-2}$ with multiplicity $a_{i-2}$. If we denote by $b_i:=a_i/a_{i-1}$, then an easy calculation based on~\eqref{mui} gives
\[
\mu_i=(\mu_{i-1}-1)b_i.
\]
Observe that the $a_i$'s satisfy a Fibonacci type recursion and the $b_i$'s are rational approximations in the continued fraction expansion of the number $\frac{1}{2}(1+\sqrt{5})$. In particular we claim that $b_i>8/5$ for $i>4$, which can be easily proved by induction using the fact that $b_0=7/3$ and $b_i=1+1/b_{i-1}$.
We prove now by induction that $\mu_i>8/3$ for each $i$. We have $\mu_0=3$ and
$\mu_{i-1}>8/3=1+1/(8/5-1)>1+1/(b_i-1)$ so that $\mu_i=(\mu_{i-1}-1)b_i>\mu_{i-1}>8/3$. 
Since $\mu_i>1$, then $\kappa(-K_{X_i})=1$ by Lemma~\ref{mu}. Thus the surface $X_{n-9}$ has the required properties.
\end{proof}

\begin{figure}[h]
\begin{center}
\includegraphics[scale=.4]{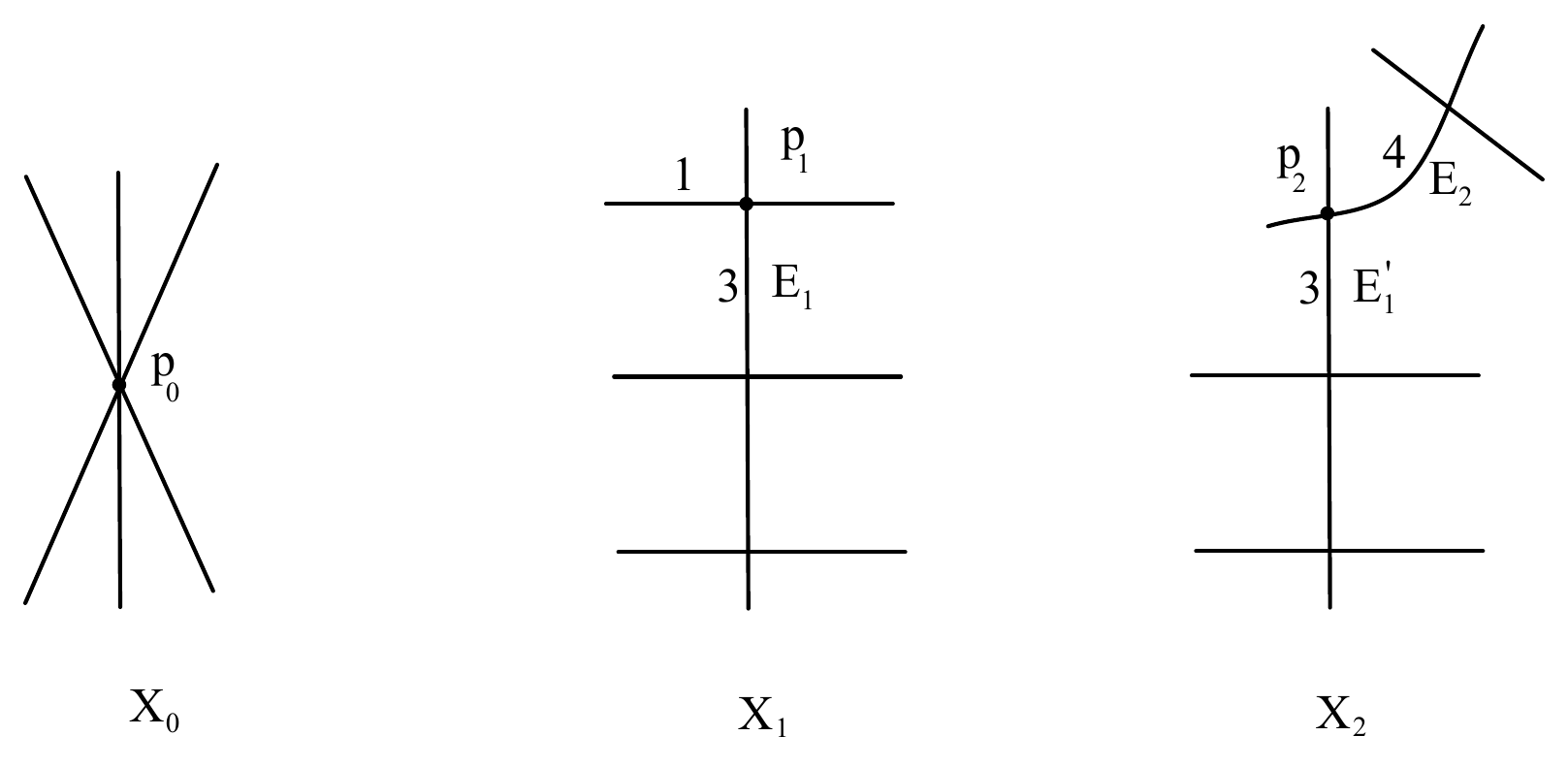}
\end{center}
\caption{The sequence of blow-ups in Proposition~\ref{exist}}\label{blow-up}
\end{figure}

We now want to relate the structure of the effective cone $\Eff(X)$ of a smooth rational surface $X$ with $\kappa(-K_X)=1$ with that of its relative minimal model $Y$.
Recall that the birational morphism $\pi: X\to Y$ induces an injective linear map $\pi^*:\Pic(Y)\otimes{\rr}\to\Pic(X)\otimes{\rr}$ which maps $\Eff(Y)$ into a linear section of $\Eff(X)$. Thus if $\Eff(X)$ is rational polyhedral, then the same is true for $\Eff(Y)$. We will show that the converse statement also holds.

\begin{lemma}\label{finite}
Let $Y$ be a smooth rational surface with $\kappa(-K_Y)=1$, $-K_Y$ nef and rational polyhedral $\Eff(Y)$. If $a,b$ are non-negative integers with $b\neq 0$, then there is a finite number of classes $[D]\in \Nef(Y)$ such that $D^2=a,\ -K_Y\cdot D=b$.
\end{lemma}
\begin{proof}
Observe that $Y$ has a minimal elliptic fibration $\varphi=\varphi_{|-mK_Y|}$ for some $m>0$ by~\cite[Prop. 5.6.1]{cd}.
Since the effective cone of $Y$ is rational polyhedral, then 
 
there are $8$ components of reducible fibers of $\varphi$ whose intersection matrix $M$ is negative definite (see Proposition~\ref{class}(ii) below). Let $f_1,\dots,f_8$ be the classes of such curves, $f$ be the class of a fiber of $\varphi$ and  $s$ be the class of a $m$-section of $\varphi$ so that $f\cdot s=m$.
Observe that the lattice $L:=\langle f,s,f_1,\dots,f_8\rangle$ has rank $10$, so that it has finite index $k$ in $\Pic(Y)$. 
Let 
\[
[D]=\alpha f+\sum_{i=1}^{8} \alpha_i f_i+\beta s\in \Pic(Y)
\]
be a nef class with $-K_Y\cdot D=b$ and $D^2=a$. Thus $b=-K_Y\cdot D=\frac{1}{m}f\cdot D=\beta$.

Since $D$ is nef and $f-f_i$ is an effective class, then $b_i= D\cdot f_i\leq D\cdot f=b$. Observe that, since $[kD]\in \Pic(Y)$, then the coefficients of $[D]$ are rational with bounded denominators. Hence $b_i$ can take a finite number of non-negative rational values. For any such choice of the $b_i$'s, the coefficients $\alpha_i$ are uniquely determined since the intersection matrix $M$ of the $f_i$'s is non singular.
Finally, the condition $D^2=a$ determines $\alpha$, since $D^2=-b^2+(\sum_i\alpha_if_i)^2+2b \sum_i\alpha_i f_i\cdot s+2\alpha b m$.
This proves that there is a finite number of classes $[D]$ as in the statement.
\end{proof}

\begin{theorem}\label{effk1}
Let $X$ be a smooth rational surface with $\kappa(-K_X)=1$ and such that its relative minimal model $Y$ has rational polyhedral effective cone. Then $\Eff(X)$ is rational polyhedral.
\end{theorem}
\begin{proof}
By Theorem \ref{equiv} it is enough to prove that $X$ contains finitely many $(-1)$-curves.
Let $\pi:X\to Y$ be the blow-down map. 
We assume that $X$ is the blow-up of $Y$ at $r$ points, possibly infinitely near, and we call $E$ the exceptional divisor of $\pi$.  Observe that we can write $E=\sum_{i=1}^r c_i E_i$, where $c_i$ are positive integers and $E_i$ are curves (not necessarily integral) such that $E_i^2=-1$ and $E_i\cdot E_j=0$ for distinct $i, j$.

Let $F$ be a $(-1)$-curve of $X$ and let $-K_X\sim N+P$ be the Zariski decomposition. Observe that $P\sim\alpha\pi^*(-K_Y)$, where $\alpha$ is a rational number with $0<\alpha<1$ by Lemma~\ref{mu}.
Since $F\cdot (-K_X)=1$, then either $F$ is a component of $N$ or $F\cdot P=0$ or
$0< F\cdot P \leq 1$. 
In the first two cases $F$ belongs to a finite set of curves, so we assume to be in the third case. Then we have
\[
F\cdot E=F\cdot(\pi^*(-K_Y)+K_X)=F\cdot \pi^*(-K_Y)-1\leq \frac{1}{\alpha}-1.
\]
Observe that we can also assume that $F\cdot E_i\geq 0$ for each $i$, since otherwise $F$ would be a component of $E_i$ and again there is only a finite set of such components.
If $D=\pi(F)$, then we can write $F$ as
\[
F=\pi^*(D)-\sum_{i=1}^ra_iE_i
\]
with $a_i=F\cdot E_i\leq F\cdot E\leq \frac{1}{\alpha}-1$ so that
$D^2\leq -1+r(\frac{1}{\alpha}-1)^2$ since $F^2=-1$.
Moreover 
\[
0<-K_Y\cdot D=\frac{1}{\alpha} P\cdot (F+\sum_ia_iE_i)\leq \frac{1}{\alpha} (1+(\frac{1}{\alpha}-1)\sum_iP\cdot E_i).
\]
Observe that either $D$ is a $(-1)$-curve of $Y$ or $D^2\geq 0$ and $D$ is a nef divisor since it is integral. In the first case there are a finite number of such $D$ by 
Theorem~\ref{equiv}. In the second case we conclude by Lemma \ref{finite} since both $D^2$ and $-K_Y\cdot D$ are bounded.
This implies that $X$ contains finitely many $(-1)$-curves.
\end{proof}
\begin{corollary}\label{infty}
Let $n\geq 10$ be an integer. Then there exists a smooth rational surface $X$ with
$\kappa(-K_X)=1$, $\rho(X)=n$ and finitely generated Cox ring $R(X)$.
\end{corollary}
\begin{proof}
Let $Y$ be a minimal elliptic surface with $4$ fibers of Kodaira type $\tilde A_2$. Such a surface exists and its effective cone is rational polyhedral by~\cite[Ex.~1.4.1]{n1}.
Observe that $Y$ contains a fiber with a triple point, thus proceding as in the proof of Proposition~\ref{exist} we construct a surface $X$ of Picard number $n$ with $\kappa(-K_X)=1$ 
whose relative minimal model is $Y$.
Since $\Eff(Y)$ is rational polyhedral then $\Eff(X)$ is rational polyhedral by Theorem~\ref{effk1}.
We now conclude by Theorem~\ref{equiv}. 
\end{proof}

\section{Smooth projective surfaces with $-K_X$ nef}
In this section we consider smooth projective 
surfaces $X$ such that $q(X)=0$, $\Eff(X)$ is rational polyhedral 
and $-K_X$ is nef. Observe that, since $-K_X$ is nef, 
then $K_X^2\geq 0$ and the integral curves with 
negative self-intersection are 
either $(-1)$ or $(-2)$-curves.
If $X$ is minimal, i.e. it does not contain 
$(-1)$-curves, then either $\rho(X)\leq 2$ or $K\equiv 0$ by 
Proposition \ref{eff-rays}.
In the first case, $X$ is either $\pp^2$ or a 
Hirzebruch surface $\mathbb{F}_n$ with $n=0,2$. 
In the second case, $X$ is either a K3 surface or 
an Enriques surface. In~\cite[Thm.~2.7, Thm.~2.10]{ahl} it is proved that
 the Cox ring of these surfaces is finitely generated 
if and only if $\Eff(X)$ is rational polyhedral. 
Moreover, the Picard lattices of those $X$ which 
admit a rational polyhedral effective cone are 
classified in a series of papers by Nikulin and Kond\=o (see ~\cite[Ex.~1.4.1]{n1} for precise references).

If $X$ is non-minimal, then it is rational and it is 
either a Hirzebruch surface $\mathbb{F}_1$ or one 
of the surfaces described in the following 
(for the proof see~\cite[Ex.~1.4.1]{n1}).
\begin{proposition}\label{class}
Let $X$ be a smooth rational surface with $-K_X$ 
nef and $\rho(X)\geq 3$. Then $\Eff(X)$ is rational 
polyhedral if and only if one of the following holds:
\begin{enumerate}
\item $K_X^2>0$ and $X$ is the minimal resolution 
of singularities of a Del Pezzo surface with Du Val 
singularities;
\item $K_X^2=0$ and any connected component 
of the set of $(-2)$-curves of $X$ is an extended 
Dynkin diagram of rank $r_i$ with $\sum_i r_i=8$. 
\end{enumerate}
\end{proposition}

The surfaces in Proposition~\ref{class} (ii)  
 have been further classified and divided in two classes.

If $\kappa(-K_X)=1$, then $\varphi_{|-mK_X|}$
is an elliptic fibration for some $m>0$.
Then (ii) is equivalent to ask that the Jacobian fibration 
of $\varphi_{|-mK_X|}$ has finite Mordell-Weil group.
In case $m=1$ the reducible fibers and the Mordell-Weil groups of 
such surfaces have been classified in \cite{d,cd}.

If $\kappa(-K_X)=0$, then $|-mK_X|$ is zero-dimensional 
for all positive $m$. There are exactly three families of
such surfaces, which have been classified in~\cite[Ex.~1.4.1]{n1}.  
 
 \begin{remark} If $X$ is a smooth rational surface  
as in Proposition~\ref{class}, then $\Eff(X)$ is generated 
by $(-1)$ and $(-2)$-curves by Proposition \ref{eff-rays}. 
In particular, if $\varphi=\varphi_{|-K_X|}$ is an elliptic fibration, 
then these curves will be the sections and the components 
of reducible fibers of $\varphi$ respectively.
 \end{remark}
\begin{theorem}\label{nfg}
Let $X$ be a smooth projective surface with $q(X)=0$ and $-K_X$ 
nef. Then $R(X)$ is finitely generated
if and only if one of the following holds:
\begin{enumerate}
\item  $X$ is the minimal resolution 
of singularities of a Del Pezzo surface with Du Val 
singularities;
\item  $\varphi_{|-mK_X|}$ is an elliptic fibration for some $m>0$ and the Mordell-Weil group of the Jacobian fibration of $\varphi_{|-mK_X|}$ is finite;
\item $X$ is either a K3-surface or an Enriques surface with finite automorphism group $\Aut(X)$.
\end{enumerate}
\end{theorem}
\begin{proof}
 If $K_X\equiv 0$, then $X$ is either a K3 or an Enriques surface and we conclude by~\cite[Thm.~2.7, Thm.~2.10]{ahl}.
If $K_X\not\equiv 0$ and $\rho(X)\leq 2$, then $X$ is either $\pp^2$ or a Hirzebruch surface $\mathbb{F}_0$, $\mathbb{F}_2$. In all these cases $X$ is in (i).

Assume now that $K_X\not\equiv 0$ and  $\rho(X)\geq 3$.
If $K_X^2>0$, then $\kappa(-K_X)=2$, 
so that $R(X)$ is finitely generated by~\cite[Thm.~2.9]{tvv}.
If $K_X^2=0$ and $\kappa(-K_X)=1$ then, by Proposition~\ref{class} and Theorem~\ref{equiv}, $R(X)$ is finitely generated if and only if we are in (ii).
If $K_X^2=0$ and  $\kappa(-K_X)=0$, then $-K_X$ is nef but not semiample since 
$h^0(-mK_X)= 1$ for any $m>0$. Thus $R(X)$ is not finitely generated by~\cite[Prop.~2.9]{hk} or \cite[Corollary 2.6]{ahl}. 
\end{proof}

\begin{corollary}\label{corell}
Let $X$ be a smooth rational surface such that 
$\varphi_{|-K_X|}: X\to\pp^1$ is an elliptic fibration. 
Then $R(X)$ is finitely generated if and only if the 
Mordell-Weil group of $\varphi_{|-K_X|}$ is finite.
\end{corollary}
 
\begin{remark}\label{non-fg}
\cite[Prop.~2.9]{hk} or Theorem~\ref{nfg} and \cite[Ex.~1.4.1]{n1} provide a negative answer 
to~\cite[Question I.3.9]{har}, since they show that 
there are rational surfaces such that the effective 
cone is rational polyhedral but the Cox ring is not 
finitely generated.
\end{remark}

\section{An example with $\kappa(-K_X)=-\infty$}

In this section we construct a surface $X$ with $\rho(X)=2$ such that $\Eff(X)$ is rational polyhedral but the Cox ring $R(X)$ is not finitely generated. The surface $X$ will be the blow-up of a smooth, very general quartic $S\subset\pp^3$ at a very general point $p$. In particular $\kappa(-K_X)=-\infty$. 
If $\pi: X\to S$ is the blow-up at $p$ with exceptional divisor $E$ and $H=\pi^*\Osh_{S}(1)$, we will show that: 
\[
\Eff(X)=\langle [E],[H-2E]\rangle,\quad
\Nef(X)=\langle [H],~ [H-2E]\rangle,\quad [H-2E]\not\in\SAmple(X),
\]
since $h^0(m(H-2E))=1$ for any $m\geq 1$.
This implies, by~\cite[Proposition 2.9]{hk} or~\cite[Corollary 2.6]{ahl}, that $R(X)$ is not finitely generated.  
\begin{remark} In the given example it is possible to prove directly, without using~\cite[Proposition 2.9]{hk}, that $R(X)$ is not finitely generated.
To any $x\in H^0(D)$ we associate its {\em degree}:
\[
[x]:=[D]\in\Pic(X).
\]
Assume that $R(X)$ is finitely generated, so that there exists $[D]\in\Nef(X)$ such that the degree $[x]$ of any generator either lives in the cone $\langle [D],[E]\rangle$ or is equal to $[H-2E]$. A class $[D']$ as in the picture is ample, since it lives in the interior of the nef cone, so that there is a section $x'\in H^0(nD')$ which is not divisible by $z\in H^0(H-2E)$, for $n$ big enough. Thus $[x']$ is a non-negative linear combination of the degrees of the generators of $R(X)$ distinct from $z$. This means that $[nD']=[x']$ lives in the cone $\langle [D],[E]\rangle$, which is a contradiction.
\end{remark}
\begin{center}
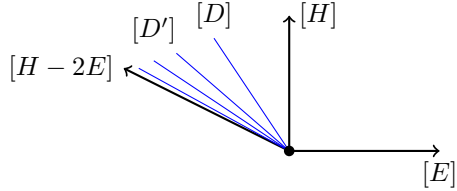
\begin{figure}[h]
\begin{tikzpicture}[scale=1]
\begin{scope}
  \draw[->,thick] (0,0) -- (2,0) node[below]{$[E]$};
  \draw[->,thick] (0,0) -- (-2.2,1.1) node[left] {$[H-2E]$};
  \draw[->,thick] (0,0) -- (0,1.8) node[right]{$[H]$};
  \foreach \x/\y in {-1/1.5,-1.5/1.3,-1.8/1.2,-2/1.1} \draw[-,color=blue] (0,0) to (\x,\y);
  \fill[black] (0,0) circle (2pt);
  \node[above] at (-1,1.5) {$[D]$};
  \node[above] at (-1.8,1.3) {$[D']$};
\end{scope}
\end{tikzpicture}
\caption{Each blue ray provides new generators of $R(X)$.}
\end{figure}
\end{center}
If $S\subset \pp^3$ is a smooth quartic surface and $p\in S$ we denote by $\pi:C\to S\cap T_pS$ the normalization at $p$ of the hyperplane section $S\cap T_pS$. Observe that, for general $p$,  $S\cap T_pS$ has a node at $p$ and $C$ is a smooth genus two curve.
\begin{lemma}\label{tor}
There exists a smooth quartic surface $S\subset\pp^3$ with $\rho(S)=1$ and a point $p\in S$ such that  $C$ is smooth of genus two and $K_{C}-q_1-q_2$ is not a torsion point of $J(C)$, where $\{q_1,q_2\}=\pi^{-1}(p)$.
\end{lemma}
\begin{proof}
Let $S_0\subset\pp^3$ be a smooth quartic surface with $\rho(S_0)=1$, $p\in S_0$ and  
$B_0=S_0\cap T_pS_0$. We denote by $B_1\subset T_pS_0$ a plane quartic with just one node at $p$ such that, if $\pi_1: C_1\to B_1$ is its normalization, then $K_{C_1}-q_{11}-q_{12}$ is not a torsion point of $J(C_1)$, where $\{q_{11},q_{12}\}=\pi_1^{-1}(p)$.
Observe that such a plane quartic $B_1$ exists because, if $C$ is a genus two curve, then the morphism $\varphi$ associated to the linear system $|K_{C}+q_1+q_2|$  is birational onto a plane quartic with just one node at  $\varphi(q_1)=\varphi(q_2)$. 

Let $f:\mathcal B\to \pp^1$ with $\mathcal B\subset T_pS_0\times \pp^1$ be  the pencil  of plane quartics  generated by $B_0=f^{-1}(0)$ and $B_1=f^{-1}(1)$. We denote by $\pi: \mathcal C \to \mathcal B$ the normalization of $\mathcal B$ along the section $\{(p,t): t\in \pp^1\}$ and by $J(\mathcal C)\to \pp^1$ its relative jacobian. We will denote by $B_t$ and $C_t$ the fiber of
 $f$ and of $f\circ \pi$ respectively over $t\in \pp^1$. 
 Let $L_t:=K_{C_t}-q_{t1}-q_{t2}$, where $\{q_{t1},q_{t2}\}=\pi^{-1}(p,t)$. Since $L_1$ is not a torsion point of $J(C_1)$ and the set of torsion  sections in $J(\mathcal C)$ is countable, then $L_t$ is not a torsion point of $J(C_t)$ outside of a countable set of $t\in\pp^1$.

Let $\mathcal{S}:=\{(S,t)\in |\Osh_{\pp^3}(4)|\times\pp^1 : B_t\subset S\}$ be the family of quartic surfaces which contain one $B_t$  and let $(S_t,t)$ be a curve in $\mathcal{S}$ which maps birationally onto $\pp^1$ and which contains $(S_0,0)$. Observe that $B_t=S_t\cap T_pS_0$ for each $t\in\pp^1$.
Since $S_0$ is smooth with $\rho(S_0)=1$, then the same is true for $S_t$ outside of a countable set of values (see for example ~\cite[Thm.~1.1]{og}). Thus there exists $t_0\in\pp^1$ such that $S_{t_0}$ is smooth with $\rho(S_{t_0})=1$ and $L_{t_0}$ is not a torsion point of $J(C_{t_0})$.
\end{proof}

\begin{proposition}
There exists a smooth quartic surface $S\subset\pp^3$ with $\rho(S)=1$ and a point $p\in S$ such that, if $\pi: X\to S$ is the blow-up at $p$, then $\Eff(X)$ is rational polyhedral but $R(X)$ is not finitely generated.
\end{proposition}
\begin{proof}
Let $S$ and $p$ be as in Lemma~\ref{tor}.
We start proving that $\Eff(X)$ is rational polyhedral. Let $E$ be the exceptional divisor of $\pi$ and $H=\pi^*\Osh_S(1)$. The class of the strict transform $C$ of $B:=S\cap T_pS$ is $[H-2E]$. Since $C$ is integral and $(H-2E)^2=0$, then $[H-2E]$ is nef.

Let $[D]:= [aC-bE]$ where $a,b$ are positive integers. If $D$ is effective, then $D\cdot C<0$ so that $h^0(D) = h^0(D-C)$ and $(a-1)C-bE$ is effective. Applying this reasoning $a$ times we deduce $h^0(-bE)>0$, which is a contradiction. Hence $[D]$  is not effective. Thus $\Eff(X)=\langle [H-2E],[E]\rangle$ is rational polyhedral.

We will now prove that $\SAmple(X)\subsetneq\Nef(X)$, so that $R(X)$ is not finitely generated by~\cite[Cor.~2.6]{ahl}.
Since $K_X\sim E$, then by adjunction formula we have $C_{|C}\sim K_C-E_{|C}\sim K_C-q_1-q_2$, where $\pi(q_1)=\pi(q_2)=p$. By Lemma~\ref{tor} this implies that $C_{|C}$ is not a torsion point of $J(C)$ or, equivalently, that $h^0(nC_{|C})=0$ for any $n>0$. From the exact sequence
\[
0\to H^0((n-1)C)\to H^0(nC)\to H^0(nC_{|C})=0,
\]
and $h^0(\Osh_X)=1$, we get that $h^0(nC)=1$ for any positive $n$. This implies that $H-2E$ is not semiample.
\end{proof}

\bibliographystyle{amsplain}

\end{document}